\documentclass[12pt]{article}
\usepackage{amssymb,amsfonts,amsmath,amsthm,amscd}
\usepackage[usenames]{color}
\usepackage[colorlinks=true,linkcolor=webgreen,filecolor=webbrown,
citecolor=webgreen]{hyperref}
\definecolor{webgreen}{rgb}{0,.5,0}
\definecolor{webbrown}{rgb}{.6,0,0}
\usepackage{chngcntr}

\def\lcm{\operatorname{lcm}}
\def\cdeg{\operatorname{cdeg}}
\def\ndeg{\operatorname{ndeg}}
\def\sd{\operatorname{sd}}
\def\Z{{\mathbb{Z}}}
\def\N{{\mathbb{N}}}

\newtheorem{theorem}{Theorem}[section]
\newtheorem{lemma}[theorem]{Lemma}
\newtheorem{cor}[theorem]{Corollary}

\newtheorem{prop}[theorem]{Proposition}

\def\3{\subset }
\def\4{\subseteq }
\def\alege#1#2{{\displaystyle{{#1}\choose{#2}}}}
\def\<{\left<}
\def\>{\right>}
\def\vsp{\vspace*{1,5mm}\\ }
\def\n{\noindent }
\def\bk{\bigskip }
\def\bit{\begin{itemize}}
\def\eit{\end{itemize}}
\def\3{\subset }
\def\4{\subseteq }

\def\alege#1#2{\Bigl(\begin{array}{c}{#1}\\ {#2}\end{array}\Bigr)}

\def\0{\leqno}
\def\a{{\alpha}}

\def\barr{\begin{array}}
\def\earr{\end{array}}
\def\dd{\displaystyle}

\def\calg{{\cal G}}

\def\frax{\dd\frac}

\def\n{\noindent }

\def\acom{\vrule height 1.4ex width .1ex depth -.1ex}
\def\8{\hbox{$\acom\!\raise1pt\hbox{$\times$}$}}

\title{\bf Cyclicity degrees of finite groups}
\author{Marius T\u arn\u auceanu, L\'{a}szl\'{o} T\'{o}th}
\date{}

\begin{document}

\maketitle

\centerline{Acta Math. Hungar. {\bf 145} (2015), 489--504}

\begin{abstract}
We introduce and study the concept of cyclicity
degree of a finite group $G$. This quantity measures the
probability of a random subgroup of $G$ to be cyclic. Explicit
formulas are obtained for some particular classes of finite
groups. An asymptotic formula and minimali\-ty/maximality results on cyclicity
degrees are also inferred.
\end{abstract}




\section{Introduction}

In the last years there has been a growing interest in the use of
probability in finite group theory. One of the most important
aspects which have been studied is the probability that two
elements of a finite group $G$ commute. It is called the {\it
commutativity degree} of $G$, and has been investigated in many
papers, as \cite{Gus1973,Les1987,Les1988,Les1989,Les1995,Les2001,Rus1979,She1975}. Inspired by
this concept, in \cite{Tar2009} the first author introduced a similar notion for the
subgroups of $G$, called the {\it subgroup commutativity degree}
of $G$. This quantity is defined by
$$\barr{lcl} \sd(G)&=&\frax1{|L(G)|^2}\,\left|\{(H,K)\in L(G)^2:
HK=KH\}\right|=\vsp &=&\frax1{|L(G)|^2}\,\left|\{(H,K)\in L(G)^2:
HK\in L(G)\}\right|\earr$$ (where $L(G)$ denotes the subgroup
lattice of $G$) and it measures the probability that two subgroups
of $G$ commute, or equivalently the probability that the product of
two subgroups of $G$ be a subgroup of $G$ (recall also the natural
generalization of $\sd(G)$, namely the {\it relative subgroup
commutativity degree} of a subgroup of $G$, introduced and studied
in \cite{Tar2011}).
\bigskip

Another probabilistic notion on $L(G)$ has been investigated in
\cite{Tar}: the {\it normality degree} of $G$. It is defined by
$$\ndeg(G)=\dd\frac{|N(G)|}{|L(G)|}$$ (where $N(G)$
denotes the normal subgroup lattice of $G$) and measures the
probability of a random subgroup of $G$ to be normal.
\bigskip

Clearly, analogous constructions can be made by replacing $N(G)$
with other remarkable posets of subgroups of $G$. One of them is the
poset of cyclic subgroups of $G$, usually denoted by $C(G)$. In this
way, one obtains the following quantity
$$\cdeg(G)=\dd\frac{|C(G)|}{|L(G)|}\hspace{1mm},$$that measures the probability of a random
subgroup of $G$ to be cyclic and will be called the {\it cyclicity
degree} of $G$. Its study is the main purpose of our paper.
\bigskip

Note that for an arbitrary finite group $G$ computing the number of
subgroups, as well as the number of cyclic subgroups, is a difficult
work. These numbers are in general unknown, excepting for few
particular classes of finite groups. We also recall the
powerful connection between the cyclic subgroup structure of a
finite group and the set of its element orders.
\bigskip

The paper is organized as follows. Some basic properties and results
on cyclicity degree are presented in Section \ref{sect_basic}. Section \ref{sect_classes} deals with
cyclicity degrees for some special classes of finite groups: abelian
groups, hamiltonian groups, $p$-groups possessing a cyclic maximal
subgroup and ZM-groups. An asymptotic formula for $\sum_{n\le x} \cdeg(\Z_n\times \Z_n)$ is also deduced.
In Section \ref{sect_min_max} we give several minimality/maximality results and conjectures on
cyclicity degrees of finite abelian $p$-groups. In the final section an problem is formulated.
\bigskip

Most of our notation is standard and will usually not be repeated
here. Elementary notions and results on lattices (respectively on
groups) can be found in \cite{Gra1978} (respectively in \cite{Hup1967,Isa2008,Suz1982}).
For subgroup lattice concepts we refer the reader to
\cite{Sch1994,Tar2003,Tar2006}.


\section{Basic properties of cyclicity degree} \label{sect_basic}

Let $G$ be a finite group. First of all, remark that the cyclicity
degree $\cdeg(G)$ satisfies the following relation
$$0<\cdeg(G)\le 1.$$

Moreover, we have $\cdeg(G)=1$ if and only if $G$
is cyclic.
\bigskip

Next we observe that several finite non-cyclic groups, as $S_3$,
satisfy the property that all their proper subgroups are cyclic.
We are able to determine all finite groups satisfying this
property.

\begin{theorem} \label{theor_char}
Let $G$ be a finite group. Then
$$\cdeg(G)=\frax{|L(G)|-1}{|L(G)|}$$
if and only if $G$ is either a certain semidirect product of a normal subgroup of order $p$ by a
cyclic subgroup of order $q^n$ {\rm(}$p$, $q$ primes, $n\in
\mathbb{N}^*:=\{1,2,\ldots\} ${\rm)}, an elementary abelian $p$-group of rank two or
the quaternion group $Q_8$.
\end{theorem}

\begin{proof} Let $G$ be a finite non-cyclic group all of whose
proper subgroups are cyclic.

Case I. Assume that $G$ is not a $p$-group. Since $G$ is not cyclic and all its Sylow subgroups
are cyclic, we infer that there is a Sylow $q$-subgroup $S\cong {\mathbb{Z}}_{q^n}$ which is not normal.
Then $N_G(S)$ is different from $G$ and consequently it is cyclic by our
hypothesis. It follows that $S$ is contained in $Z(N_G(S))=N_G(S)$. By the
Burnside normal $p$-complement theorem (see \cite[Th.\ 5.13]{Isa2008})
we obtain that $G$ has a normal $q$-complement $T$ (note that $T$ is also
cyclic). Then $TS=G$, i.e., $G$ is a semidirect product of a cyclic normal
subgroup by a cyclic subgroup of order $q^n$ (G is, in fact, a metacyclic
group). Now, the proof is completed by the remark that $T$ can be chosen to
be of a prime order $p$ (we can replace it by a normal subgroup $T'\leq T$ of
order $p$).

Case II. Now assume that $G$ is a $p$-group. Let $M_1$ be a minimal
normal subgroup of $G$. If there is $M_2\in L(G)$, $M_2\neq M_1$,
with $|M_2|\hspace{1mm}=p$, then $G$ will contain the noncyclic
subgroup $M_1M_2\cong \mathbb{Z}_p \times \mathbb{Z}_p$. It follows
that $G=M_1M_2$ and so $G\cong \mathbb{Z}_p \times \mathbb{Z}_p$. If
$M_1$ is the unique subgroup of order $p$ in $G$, then, by \cite[vol.\ II, eq.\ (4.4)]{Hup1967},
$G$ is isomorphic to a generalized quaternion
group $$Q_{2^n}=\langle a,b \mid a^{2^{n-2}}= b^2, a^{2^{n-1}}=1,
b^{-1}ab=a^{-1}\rangle, \hspace{1mm}n\geq 3.$$Let $H=\langle
a^2,b\rangle\cong Q_{2^{n-1}}$. Clearly, $H$ is a proper non-cyclic
subgroup of $G$ for $n\geq 4$. Hence $n=3$, i.e., $G\cong Q_8$.
\end{proof}

In what follows assume that $G$ and $G'$ are two finite groups. It
is obvious that if $G\cong G'$, then $\cdeg(G)=\cdeg(G')$. In
particular, we infer that any two conjugate subgroups of a finite
group have the same cyclicity degree. We remark that the above conclusion
remains true even in the case when $G$ and $G'$ are only lattice-isomorphic ($L$-isomorphic), i.e.,
there is an isomorphism from $L(G)$ to $L(G')$, because lattice-isomorphisms preserve cyclic
subgroups by \cite[Th.\ 1.2.10]{Sch1994}.

\bigskip
By a direct calculation, one obtains
$$\cdeg(S_3 \times \mathbb{Z}_2)=\frax{5}{8} \ne \frax{5}{6}=\cdeg(S_3)\cdeg(\mathbb{Z}_2)$$
and therefore in general we do not have $\cdeg(G\times
G')=\cdeg(G)\cdeg(G')$. It is clear that a sufficient condition for this equality to hold is
$$\gcd(|G|,|G'|)=1,$$ that is $G$ and
$G'$ are of coprime orders. This can be extended to arbitrary finite
direct products.

\begin{prop} \label{prop_coprime}
Let $(G_i)_{1\le i\le k}$ be a family of finite groups
having coprime orders. Then
$$\cdeg(\prod_{i=1}^k G_i)=\prod_{i=1}^k \cdeg(G_i).$$
\end{prop}

The following immediate consequence of Proposition
\ref{prop_coprime} shows that computing the cyclicity degree of a
finite nilpotent group is reduced to finite $p$-groups.

\begin{cor} \label{cor_Sylow} If $G$ is a finite nilpotent group and
$(G_i)_{1\le i\le k}$ are the Sylow subgroups of $G$, then
$$\cdeg(G)=\prod_{i=1}^k \cdeg(G_i).$$
\end{cor}

In order to determine the cyclicity degree of a finite group $G$ it
is essential to find the number of its cyclic subgroups and this is
strongly connected with the number of elements of certain orders in
$G$. For every divisor $d$ of $n$ let $c(d)=|\{H\in C(G): |H|=d \}|$.
Then
\begin{equation} \label{eq_G}
|G|= \sum_{d\mid n}\, c(d) \varphi(d),
\end{equation}
where $\varphi$ is Euler's totient function, and
\begin{equation} \label{eq_C_G}
|C(G)|=\dd \sum_{d\mid n}\, c(d).
\end{equation}

In many particular cases the above two identities lead to a precise
expression of $|C(G)|$. For example, if $G=\Z_p^k$ is the elementary abelian
$p$-group of rank $k$ (that is $|G|=p^k$), then we can easily obtain
$c(1)=1$, $c(p)=\frax{p^k-1}{p-1}$ and $c(p^2)=c(p^3)= \ldots =
c(p^k)=0$. So, $|C(\mathbb{Z}_p)|=2$ and for $k\ge 2$ we have
$$|C(\mathbb{Z}_p^k)|=2+p+p^2+\ldots+p^{k-1}.$$

Furthermore, as it is well known, $|L(\mathbb{Z}_p^k)|=\sum_{j=0}^k \left[{k \atop j} \right]_p$ in terms of the Gaussian coefficients.
This can be written in an alternate form, see \cite[Prop.\ 2.12]{Tar2007}), giving the following identity for $\cdeg(\mathbb{Z}_p^k)$.

\begin{prop} \label{cdeg_Z_p_k} The cyclicity degree of the elementary abelian $p$-group
$\mathbb{Z}_p^k$ ($k\ge 2$) is given by
$$\cdeg(\mathbb{Z}_p^k)=\frax{2+p+p^2+\ldots+p^{k-1}}{2+\displaystyle\sum_{\a=1}^{k-1}
\displaystyle\sum_{1\le i_1<i_2<\ldots <i_\a\le k}p^{i_1+i_2+\ldots
+i_\a-\frac{\a(\a+1)}2}}\hspace{1mm}.$$
\end{prop}

In particular,
\begin{equation*}
\cdeg(\mathbb{Z}_p)=1, \quad \cdeg(\mathbb{Z}_p^2)=\frac{p+2}{p+3},
\quad \cdeg(\mathbb{Z}_p^3)=\frac{1}{2},
\end{equation*}
\begin{equation*}
\cdeg(\mathbb{Z}_p^4)=\frac{p^3+p^2+p+2}{p^4+3p^3+4p^2+3p+5}.
\end{equation*}

Observe that $\lim_{p\to \infty} \cdeg(\mathbb{Z}_p^4) =0$.
This shows the following property.

\begin{cor} $\inf\{\,\cdeg(G): G \, \text{\rm is a finite group} \}=0.$
\end{cor}


\section{Cyclicity degrees for some \\ classes of finite groups} \label{sect_classes}

\counterwithout{theorem}{section}
\counterwithin{theorem}{subsection}

In this section we determine explicitly the cyclicity degree of
several finite groups. The most significant results are obtained for
abelian groups, hamiltonian groups, $p$-groups possessing a cyclic
maximal subgroup and ZM-groups.

\subsection{The cyclicity degree of abelian groups} \label{subsect_abelian}

According to Proposition \ref{prop_coprime}, our study can be reduced
to finite abelian $p$-groups. However, we present first the
next result concerning the group $\Z_m \times \Z_n$. It was proved in \cite[Th.\ 3 and 5]{HHTW2012} that for every $m,n\in\N^*$
the number of cyclic subgroups of $\Z_m \times \Z_n$ and the total number of its subgroups is given by the identities
\begin{equation} \label{C_m_n}
cs(m,n):=|C(\Z_m \times \Z_n)| = \sum_{a\mid m, \, b\mid n}
\varphi(\gcd(a,b))
\end{equation}
and
\begin{equation} \label{L_m_n}
s(m,n):=|L(\Z_m \times \Z_n)| = \sum_{a\mid m, \, b\mid n}
\gcd(a,b).
\end{equation}

The identity \eqref{L_m_n} is a special case of the more general result of \cite{Cal1987}. We deduce  the following formula.

\begin{theorem} \label{theor_cdeg_2} The cyclicity degree of the group $\Z_m \times \Z_n$
($m,n\in \N^*$) is
\begin{equation} \label{cdeg_rank_2}
\cdeg(\Z_m \times \Z_n) = \frac{cs(m,n)}{s(m,n)}
\end{equation}
where $cs(m,n)$ and $s(m,n)$ are given by \eqref{C_m_n} and \eqref{L_m_n}, respectively.
\end{theorem}

We note that for every $n_1,\ldots, n_k\in \N^*$,
\begin{equation} \label{C_gen}
|C(\Z_{n_1} \times \cdots \times \Z_{n_k})| = \sum_{a_1\mid n_1,\ldots, a_k\mid n_k}
\frac{\varphi(a_1)\cdots \varphi(a_k)}{\varphi(\lcm(a_1,\ldots,a_k))},
\end{equation}
which was proved using the orbit counting lemma (Burnside's lemma) in \cite{Tot2011} and by simple
number-theoretic arguments in \cite{Tot2012}.

A formula for $\cdeg(\Z_m \times \Z_n\times \Z_r)$ ($m,n,r\in \N^*$), which is similar to \eqref{cdeg_rank_2} follows at once from
\eqref{C_gen} applied for $k=3$ and from the identity derived for $|L(\Z_m \times \Z_n \times \Z_r)|$ in \cite[Th.\ 2.2]{HamTot2013}.

Consider now finite abelian $p$-groups. By the fundamental theorem of finitely generated abelian groups, such a
group has a direct decomposition of type
$$\mathbb{Z}_{p^{\a_1}}\times\mathbb{Z}_{p^{\a_2}}\times \cdots \times\mathbb{Z}_{p^{\a_k}},$$
where $p$ is a prime and $1\le\a_1\le\a_2\le...\le\a_k.$
\bigskip

Assume that $k=2$. Then we have the next result.

\begin{cor} \label{cor_cdeg_2} The cyclicity degree of the abelian $p$-group $\mathbb{Z}_{p^{\a_1}}\times \mathbb{Z}_{p^{\a_2}}$ is
given by the following identity:
\begin{equation*}
\cdeg(\mathbb{Z}_{p^{\a_1}}\times\mathbb{Z}_{p^{\a_2}})
\end{equation*}
\begin{equation*}
= \frac{(\a_2-\a_1+1)p^{\a_1+2}-2(\a_2-\a_1)p^{\a_1+1}+(\a_2-\a_1-1)p^{\a_1}-2p+2}
{(\a_2-\a_1+1)p^{\a_1+2} - (\a_2-\a_1-1)p^{\a_1+1}-(\a_1+\a_2+3)p+ (\a_1+\a_2+1)}.
\end{equation*}
\end{cor}

\begin{proof} Follows from Theorem \ref{theor_cdeg_2} by computing the values of the corresponding sums. Alternatively, one can use the
explicit formulas
\begin{equation*}
|C(\Z_{p^{\a_1}} \times \Z_{p^{\a_2}})|
= 2+2p+\ldots +2p^{\a_1-1} +(\a_2-\a_1+1)p^{\a_1},
\end{equation*}
\begin{equation*}
|L(\Z_{p^{\a_1}} \times \Z_{p^{\a_2}})|
\end{equation*}
\begin{equation*}
=\frac{1}{(p{-}1)^2}((\a_2{-}\a_1{+}1)p^{\a_1{+}2}{-}(\a_2{-}\a_1{-}1)p^{\a_1{+}1}{-}(\a_1{+}\a_2{+}3)p{+}(\a_1{+}\a_2{+}1))
\end{equation*}
obtained in \cite[Th. 3.3, 4.2]{Tar2010} by different arguments.
\end{proof}

For an arbitrary $k$, a similar formula is more difficult to obtain. It is well-known that
$|L(\mathbb{Z}_{p^{\a_1}}\times\mathbb{Z}_{p^{\a_2}}\times \cdots \times\mathbb{Z}_{p^{\a_k}})|$
can recursively be computed. Explicit formulas are known only in certain particular cases,
as that in Proposition \ref{cdeg_Z_p_k}. Another example of an explicit formula is that for $|L(\Z_{p^{\a}} \times \Z_{p^{\a}}
\times \Z_{p^{\a}})|$ ($\a\in \N^*$), given in \cite[Remark\
2.1]{HamTot2013}. At the same time,
$$|C(\mathbb{Z}_{p^{\a_1}}\times \mathbb{Z}_{p^{\a_2}}\times \cdots \times\mathbb{Z}_{p^{\a_k}})|=
\dd\sum_{\a=1}^{\a_k}\frax{p^\a h^{k-1}_p(\a)-p^{\a-1}h^{k-1}_p(\a-1)}{p^\a-p^{\a-1}}\hspace{1mm},$$where
$$h^{k-1}_p(\a)=\left\{\barr{lll}
p^{(k-1)\a},&\mbox{ if }&0\le\a\le\a_1;\vspace*{1,5mm}\\
p^{(k-2)\a+\a_1},&\mbox{ if }&\a_1\le\a\le\a_2;\\
\vdots\\
p^{\a_1+\a_2+\ldots +\a_{k-1}},&\mbox{ if
}&\a_{k-1}\le\a\earr\right.$$
(see \cite[Th.\ 4.3]{Tar2010}). More precisely, one obtains
$$|C(\mathbb{Z}_{p^{\a_1}}\times\mathbb{Z}_{p^{\a_2}}\times \cdots \times\mathbb{Z}_{p^{\a_k}})|=1+(\a_k-\a_{k-1})p^{\a_1+\a_2+\ldots+\a_{k-1}}+$$
$$+\frax{1}{p-1}\dd\sum_{i=1}^{k-1}p^{\a_1+\a_2+\cdots+\a_{i-1}}\frax{p^{k-i+1}}{p^{k-i}}\left(p^{(k-i)\a_i}-p^{(k-i)\a_{i-1}}\right).$$
\smallskip

Since the cyclic subgroup structure of a finite group is preserved
by $L$-isomorphisms, we remark that the above formulas can be used
to compute the cyclicity degree of groups which are $L$-isomorphic
with finite abelian groups. For example, the cyclicity degrees of
$P$-groups (see \cite[Sect.\ 2.2]{Sch1994}) can be found by
Proposition \ref{cdeg_Z_p_k}.

We close this subsection with the following properties. According to Proposition \ref{prop_coprime} the function
$n\mapsto f(n):=\cdeg(\Z_n \times \Z_n)\in (0,1]$ ($n\in \N^*$) is multiplicative, that is $f(mn)=f(m)f(n)$ for every
$m,n\in \N^*$ with $\gcd(m,n)=1$. Furthermore, by Corollary \ref{cor_cdeg_2} for every prime power $p^{\a}$ ($\a\in \N^*$),
\begin{equation*}
f(p^{\a})= \frac{p^{\a+2}-p^{\a}-2p+2}
{p^{\a+2} + p^{\a+1}-(2\a+3)p+ (2\a+1)}.
\end{equation*}

Remark that $\lim_{p\to \infty} f(p^{\a})=1$ for every fixed $\a \in \N^*$ and $\lim_{\a \to \infty} f(p^{\a})=1-1/p$ for every fixed prime $p$.
Since the series taken over the primes
\begin{equation*}
\sum_p \frac{1-f(p)}{p}= \sum_p \frac1{p(p+3)}
\end{equation*}
is convergent, it follows from the theorem of Delange (see, e.g., \cite{Pos1988}) that the function $f$ has a non-zero mean value $M$ given by
\begin{equation*}
M:=\lim_{x\to \infty} \frac1{x} \sum_{n\le x} f(n) = \prod_p \left(1-\frac1{p}\right) \sum_{a=0}^{\infty} \frac{f(p^a)}{p^a},
\end{equation*}
the product being over the primes. Here $M\approx 0.742$.

More exactly, the following asymptotic formula holds, for which we give a self contained proof.

\begin{theorem}
\begin{equation*}
\sum_{n\le x} \cdeg(\Z_n \times \Z_n) =  M x  + O(\log^3 x).
\end{equation*}
\end{theorem}

\begin{proof} Let $f(n)=\sum_{d\mid n} g(d)$ ($n\in \N^*$), that is $g=\mu*f$ in terms of the Dirichlet convolution, where $\mu$ is the
M\"{o}bius function. Here $g(p^a)=f(p^a)-f(p^{a-1})$ ($a\in \N^*$), in particular,
\begin{equation*}
g(p)= - \frac1{p+3}, \quad g(p^2)= -\frac{3p+4}{(p+3)(p^2+3p+5)}.
\end{equation*}

It follows that $|g(p)|< \frac1{p}$, $|g(p^2)|< \frac{3}{p^2}$ for every prime $p$, and direct computations show that
$|g(p^a)|<\frac{2a-1}{p^a}$ for every prime power $p^a$ ($a\in \N^*$). Hence, $|g(n)|\le \tau(n^2)/n$ for every
$n\in \N^*$, where $\tau(k)$ stands for the number of positive divisors of $k$.

Now,
\begin{equation*}
\sum_{n\le x} f(n)= \sum_{de\le x} g(d)= \sum_{d\le x} \sum_{e\le x/d} 1= \sum_{d\le x} \left( x/d+O(1)\right)
\end{equation*}
\begin{equation*}
=x \sum_{d=1}^{\infty} \frac{g(d)}{d} + O\left(x\sum_{d>x} \frac{|g(d)|}{d}\right) + O\left(\sum_{d\le x} |g(d)|\right)
\end{equation*}
\begin{equation*}
=M x  + O\left(x\sum_{d>x} \frac{\tau(d^2)}{d^2}\right) + O\left(\sum_{d\le x} \frac{\tau(d^2)}{d}\right),
\end{equation*}
where in the main term the coefficient of $x$ is $M$ by the Euler's product formula. It is known that $\sum_{n\le x}
\tau(n^2)= cx\log^2 x+O(x\log x)$ with a certain constant $c$ and partial summation shows that
$\sum_{n>x} \tau(n^2)/n^2=O((\log^2 x)/x)$, $\sum_{n\le x} \tau(n^2)/n=O(\log^3 x)$.
\end{proof}

\subsection{The cyclicity degree of hamiltonian groups}

An important class of finite groups which are closely connected to
abelian groups consists of hamiltonian groups, that is nonabelian
groups all of whose subgroups are normal. The structure of such a
group $H$ is well-known, namely $$H\cong Q_8\times\mathbb{Z}_2^n
\times A$$where $A$ is an abelian group of odd order. By Proposition \ref{prop_coprime} we infer that
$$\cdeg(H)=\cdeg(Q_8 \times
\mathbb{Z}_2^n) \cdeg(A),$$ which shows that the computation of
$\cdeg(H)$ is reduced to the computation of $\cdeg(Q_8 \times
\mathbb{Z}_2^n)$. The number of subgroups of $Q_8 \times
\mathbb{Z}_2^n$ has been determined in \cite{Tar2013}:
$$|L(Q_8\times\mathbb{Z}_2^n)|{=}\hspace{1mm}b_{n,2}{=}\hspace{1mm}2^{n+2}{+}1{+}8
\sum_{\a=0}^{n-2}(2^{n-\a}{-}2^{2\a+1}{+}2^{\a})a_{\a,2}{+}2^{n+2}a_{n-1,2}{+}a_{n,2},$$where
$a_{\a,2}=|L(\mathbb{Z}_2^{\a})|$, for all $\a\in\mathbb{N}^*$.
Moreover, the equalities \eqref{eq_G} and \eqref{eq_C_G} become in this case
$$2^{n+3}=\dd\sum_{d\mid 2^{n+3}}\hspace{1mm} c(d) \varphi(d) \hspace{1mm}\mbox{
and } |C(Q_8\times\mathbb{Z}_2^n)|=\dd\sum_{d\mid
2^{n+3}}\hspace{1mm} c(d),$$
respectively. Since $c(1)=1$ and
$c(2^i)=0$, for all $i\geq 3$, one obtains
$$2^{n+3}=1+c(2)+2c(4) \hspace{1mm}\mbox{ and }
|C(Q_8 \times\mathbb{Z}_2^n)|=1+c(2)+c(4).$$

These show that
$$|C(Q_8 \times\mathbb{Z}_2^n)|=2^{n+3}-c(4).$$

Clearly, the cyclic subgroups of order 4 of $Q_8 \times \mathbb{Z}_2^n$ are of
type $\langle(x,y)\rangle$, where $x\in Q_8$ has order 4 and $y\in
\mathbb{Z}_2^n$ is arbitrary. We infer that $c(4)=3\cdot 2^n$ and
so $$|C(Q_8 \times\mathbb{Z}_2^n)|=5\cdot 2^n.$$

Hence we have proved the following result.

\begin{theorem} The cyclicity degree of the
hamiltonian group $H\cong Q_8 \times \mathbb{Z}_2^n \times A$ is
given by the following equality:
$$ \cdeg(H)=\frax{5\cdot 2^n}{b_{n,2}}\hspace{1mm} \cdeg(A).$$
\end{theorem}

\subsection{The cyclicity degree of finite $p$-groups\\ possessing a cyclic maximal subgroup}

Let $p$ be a prime, $n \ge 3$ be an integer and denote by $\calg$
the class consis\-ting of all finite $p$-groups of order $p^n$
having a maximal subgroup which is cyclic. Obviously, $\calg$
contains the finite abelian $p$-groups of type $\mathbb{Z}_p
\times \mathbb{Z}_{p^{n-1}}$ whose cyclicity degrees have been
computed in Section \ref{subsect_abelian}, but some finite nonabelian $p$-groups
belong to $\calg$, too. They are exhaustively described in \cite[Vol.\ II, Th.\ 4.1]{Suz1982}: a nonabelian group is contained in $\calg$ if
and only if it is isomorphic to \bit\item[--] $M(p^n)=\langle
x,y\mid x^{p^{n-1}}=y^p=1,\ y^{-1}x y=x^{p^{n-2}+1}\rangle$\eit
when $p$ is odd, or to one of the following groups \bit\item[--]
$M(2^n)\ (n \ge 4),$
\item[--] the dihedral group $D_{2^n}$,
\item[--] the generalized quaternion group
$$Q_{2^n}=\langle x,y\mid x^{2^{n-1}}=y^4=1,\ yxy^{-1}=x^{2^{n-1}-1}\rangle,$$
\item[--] the quasi-dihedral group
$$S_{2^n}=\langle x,y\mid x^{2^{n-1}}=y^2=1,\ y^{-1}xy=x^{2^{n-2}-1}\rangle\ (n \ge 4)$$\eit
when $p=2$. \bk

In the following the cyclicity degrees of the above $p$-groups
will be determined. We recall first the explicit formulas for the
total number of subgroups of these groups (see \cite[Lemma 3.2.1]{Tar2011}).

\begin{lemma} The following equalities hold: \bit
\item[\rm a)] $|L(M(p^n))|=(1+p)n+1-p \hspace{1mm},$
\item[\rm b)] $|L(D_{2^n})|=2^n+n-1 \hspace{1mm},$
\item[\rm c)] $|L(Q_{2^n})|=2^{n-1}+n-1 \hspace{1mm},$
\item[\rm d)] $|L(S_{2^n})|=3 \cdot 2^{n-2}+n-1 \hspace{1mm}.$
\eit
\end{lemma}

Computing the cyclic subgroups of groups in $\calg$ is facile by
using their maximal subgroup structure and the well-known
Inclusion-Exclusion Principle (IEP, in short).
\bigskip

$M(p^n)$ has $p+1$ maximal subgroups: $M_0=\langle x\rangle$,
$M_1=\langle xy\rangle$, ..., $M_{p-1}=\langle x^{p-1}y\rangle$
and $M_p=\langle x^p,y\rangle$. Moreover, $M_i\cong
\mathbb{Z}_{p^{n-1}}$, for $i=0,1, ..., p-1$, $M_p=\langle x^p,y
\rangle\cong \mathbb{Z}_{p^{n-2}}\times\mathbb{Z}_p$ and any
intersection of at least two distinct such subgroups is isomorphic
to $\mathbb{Z}_{p^{n-2}}$. We infer that
$$\hspace{-86mm}|C(M(p^n))|=|\bigcup_{i=0}^p C(M_i)|=$$
$$\hspace{20mm}=\dd\sum_{i=0}^p |C(M_i)|-\hspace{-5mm}\dd\sum_{0\leq i_1<i_2\leq p}\hspace{-5mm}|C(M_{i_1}\cap M_{i_2})|+\ldots+(-1)^p\,|\bigcap_{i=0}^p C(M_i)|=$$
$$\hspace{22mm}=p\,|C(\mathbb{Z}_{p^{n-1}})|{+}|C(\mathbb{Z}_{p^{n-2}}\times\mathbb{Z}_p)|{-}|C(\mathbb{Z}_{p^{n-2}})|\dd\sum_{i=2}^{p+1}(-1)^i\alege{p+1}{i}\hspace{-1mm}=$$
$$\hspace{-8mm}=np+(n-2)p+2-(n-1)p=(n-1)p+2,$$which leads to the
following theorem.

\begin{theorem} The cyclicity degree of $M(p^n)$ is
$$\cdeg(M(p^n))=\frax{(n-1)p+2}{(n-1)p+n+1}\hspace{1mm}.$$
\end{theorem}

\begin{cor} We have: \bit
\item[\rm a)] $\dd\lim_{n\to\infty} \cdeg(M(p^n))=\frax{p}{p+1} \hspace{1mm},$ for every fixed prime $p$.
\item[\rm b)] $\dd\lim_{p\to\infty} \cdeg(M(p^n))=1.$
\eit
\end{cor}

The group $D_{2^n}$ has 3 maximal subgroups: $M_0=\langle
x\rangle\cong\mathbb{Z}_{2^{n-1}}$, $M_1=\langle x^2,y\rangle\cong
D_{2^{n-1}}$ and $M_2=\langle x^2,xy\rangle\cong D_{2^{n-1}}$. By
applying IEP, one obtains the recurrence relation
$$|C(D_{2^n})|=2|C(D_{2^{n-1}})|+2-n,$$which implies
that $$|C(D_{2^n})|=2^{n-1}+n.$$

So, we have:

\begin{theorem} The cyclicity degree of $D_{2^n}$
is $$\cdeg(D_{2^n})=\frax{2^{n-1}+n}{2^n+n-1}\hspace{1mm}.$$
\end{theorem}

\begin{cor}
$\dd\lim_{n\to\infty} \cdeg(D_{2^n})=\frax{1}{2}\hspace{1mm}.$
\end{cor}

$Q_{2^n}$ has 3 maximal subgroups: $M_0\cong\mathbb{Z}_{2^{n-1}}$
and $M_1, M_2\cong Q_{2^{n-1}}$. By applying again IEP, we find
the following recurrence relation
$$|C(Q_{2^n})|=2|C(Q_{2^{n-1}})|+2-n,$$proving
that $$|C(Q_{2^n})|=2^{n-2}+n.$$

In this way, one obtains:

\begin{theorem} The cyclicity degree of $Q_{2^n}$
is $$\cdeg(Q_{2^n})=\frax{2^{n-2}+n}{2^{n-1}+n-1}\hspace{1mm}.$$
\end{theorem}

\begin{cor}
$\dd\lim_{n\to\infty} \cdeg(Q_{2^n})=\frax{1}{2}\hspace{1mm}.$
\end{cor}

$S_{2^n}$ has 3 maximal subgroups: $M_0\cong\mathbb{Z}_{2^{n-1}}$,
$M_1\cong D_{2^{n-1}}$ and $M_2\cong Q_{2^{n-1}}$. In this case
IEP leads directly to an explicit formula for the number of cyclic
subgroups of $S_{2^n}$, namely
$$|C(S_{2^n})|=|C(\mathbb{Z}_{2^{n-1}})|+|C(D_{2^{n-1}})|+|C(Q_{2^{n-1}})|-2|C(\mathbb{Z}_{2^{n-2}})|=$$
$$\hspace{-58mm}=3\cdot2^{n-3}+n.$$

So, we get the following theorem.

\begin{theorem} The cyclicity degree of $S_{2^n}$
is
$$\cdeg(S_{2^n})=\frax{3\cdot2^{n-3}+n}{3\cdot2^{n-2}+n-1}\hspace{1mm}.$$
\end{theorem}

\begin{cor}
$\dd\lim_{n\to\infty} \cdeg(S_{2^n})=\frax{1}{2}\hspace{1mm}.$
\end{cor}

We end this subsection by observing that the cyclicity degree of
any finite nilpotent group whose Sylow subgroups belong to $\calg$
can explicitly be calculated, in view of Corollary \ref{cor_Sylow}.

\subsection{The cyclicity degree of ZM-groups}

Recall that a ZM-group is a finite group all of whose Sylow
subgroups are cyclic. By \cite{Hup1967} such a group is of type
$${\rm ZM}(m,n,r)=\langle a, b \mid a^m = b^n = 1, \hspace{1mm}b^{-1} a b
= a^r\rangle,$$
where the triple $(m,n,r)$ satisfies the conditions
$${\rm gcd}(m,n) = {\rm gcd}(m, r-1) = 1 \quad \text{and} \quad
r^n \equiv 1 \hspace{1mm}({\rm mod}\hspace{1mm}m).$$

It is clear that $|{\rm ZM}(m,n,r)|=mn$, ${\rm
ZM}(m,n,r)'\hspace{1mm}=\hspace{1mm}\langle a \rangle$ (therefore
we have $|{\rm ZM}(m,n,r)'|=m$) and ${\rm ZM}(m,n,r)/{\rm
ZM}(m,n,r)'$ is cyclic of order $n$. The subgroups of ${\rm
ZM}(m,n,r)$ have been completely described in \cite{Cal1987}. Set
$$L=\left\{(m_1,n_1,s)\in\mathbb{N^*}^2\times \N :
m_1\mid m,\hspace{1mm} n_1\mid n,\hspace{1mm} 0\leq s\leq m_1-1,\hspace{1mm}
m_1\mid s\frac{r^n-1}{r^{n_1}-1}\right\}.$$

Then there is a bijection between $L$ and the subgroup lattice $L({\rm
ZM}(m,n,r))$ of ${\rm ZM}(m,n,r)$, namely the function that maps a
triple $(m_1,n_1,s)\in L$ into the subgroup $H_{(m_1,n_1,s)}$
defined by
$$H_{(m_1,n_1,s)}=\bigcup_{k=1}^{\frac{n}{n_1}}\alpha(n_1,
s)^k\langle a^{m_1}\rangle=\langle a^{m_1},\alpha(n_1,
s)\rangle,$$
where $\alpha(x, y)=b^xa^y$, for all $0\leq x<n$ and
$0\leq y<m$. Notice that we have
\begin{equation} \label{eq_L_ZM}
|L({\rm ZM}(m,n,r))|=|L|= \sum_{m_1 \mid m} \sum_{n_1 \mid n}\,
\gcd \left(m_1,\frac{r^n-1}{r^{n_1}-1}\right).
\end{equation}

On the other hand, we easily infer that a subgroup $H_{(m_1,n_1,s)}\in L({\rm
ZM}(m,n,r))$ is cyclic if and only if $\frax{m}{m_1}\mid
r^{n_1}-1$. This shows that
$$C({\rm ZM}(m,n,r))=\left\{H_{(m_1,n_1,s)}\in L({\rm
ZM}(m,n,r)): (m_1,n_1,s)\in
L'\right\},$$
where
$$L'=\left\{(m_1,n_1,s)\in L : \frax{m}{m_1}\mid
r^{n_1}-1\right\}.$$

Hence
\begin{equation} \label{eq_C_ZM}
|C({\rm ZM}(m,n,r))|=|L'|= \sum_{m_1 \mid m} \sum_{\substack{n_1 \mid n\\ m/m_1 \mid
r^{n_1}-1}} \, \gcd \left(m_1,\frac{r^n-1}{r^{n_1}-1}\right)
\end{equation}
and the following result holds.

\begin{theorem} The cyclicity degree of the
ZM-group ${\rm ZM}(m,n,r)$ is
\begin{equation*}
\cdeg({\rm ZM}(m,n,r))=\frac{|L'|}{|L|},
\end{equation*}
where $|L'|$  and $|L|$ are given by the identities \eqref{eq_C_ZM} and \eqref{eq_L_ZM},
respectively.
\end{theorem}

Simple explicit formulas for $\cdeg({\rm ZM}(m,n,r))$ can be given in several particular cases. One of them is obtained by taking $n=2$, $m\equiv 1
\hspace{1mm}{\rm(mod\hspace{1mm} 2)}$ and $r=-1$, that is for the
dihedral group $D_{2m}$. Then it is easy to see that \eqref{eq_C_ZM} and \eqref{eq_L_ZM} reduce to
$$|C(D_{2m})|=m+\tau(m),$$
and
$$|L(D_{2m})|=\tau(m)+\sigma(m),$$
where $\tau(m)$ and $\sigma(m)$ denote the number and the sum of all divisors of $m$, respectively. These formulas are known, they hold true
also for $m$ even and can be proved by a direct counting of the subgroups of $D_{2m}$. We obtain the following corollary.

\begin{cor} The cyclicity degree of the
dihedral group $D_{2m}$ ($m\in \N^*$) is given by the identity
\begin{equation*}
\cdeg(D_{2m})= \frac{m+\tau(m)}{\tau(m)+\sigma(m)}.
\end{equation*}
\end{cor}

{\bf Remark.} The above formula remains true for arbitrary
primes $n$, not only for $n=2$.

\counterwithout{theorem}{subsection}
\counterwithin{theorem}{section}

\section{Some minimality/maximality results\\ on cyclicity degrees} \label{sect_min_max}

As we already have seen, computing cyclic subgroups and cyclicity
degrees of abelian groups is reduced to abelian $p$-groups. In
this section we are interested to study when for an abelian
$p$-group $\mathbb{Z}_{p^{\a_1}}\times\mathbb{Z}_{p^{\a_2}}\times \cdots\times\mathbb{Z}_{p^{\a_k}}$ ($\a_1 \leq \a_2 \leq \ldots \leq \a_k$)
of a given order $p^n$ (that is, $\sum_{i=1}^k \a_i=n$) the number
of cyclic subgroups and the cyclicity degree are minimal/maximal.
\bigskip

We suppose first that $k=2$. Then we have (cf. Section \ref{subsect_abelian})
$$|C(\mathbb{Z}_{p^{\a_1}}\times\mathbb{Z}_{p^{\a_2}})|=2+2p+\ldots+2p^{\a_1-1}+(\a_2-\a_1+1)p^{\a_1}=$$
$$\hspace{27,5mm}=2+2p+\ldots+2p^{\a_1-1}+(n-2\a_1+1)p^{\a_1},$$
in view of the equality $\a_1+\a_2=n$. By studying the above expression as a
function in $\a_1$, we easily infer that it is strictly
increasing. Therefore
$|C(\mathbb{Z}_{p^{\a_1}}\times\mathbb{Z}_{p^{\a_2}})|$ is minimal
for $\a_1=1$ and maximal for $\a_1=[n/2]$ (in other words, $\a_1$
and $\a_2$ tend to be equal).
\bigskip

In order to study the cyclicity degree of
$\mathbb{Z}_{p^{\a_1}}\times\mathbb{Z}_{p^{\a_2}}$, we remark that
the formula in Theorem \ref{theor_cdeg_2} can be rewritten as

$$\hspace{-90mm} \cdeg(\mathbb{Z}_{p^{\a_1}}\times\mathbb{Z}_{p^{\a_2}})\hspace{-1mm}=$$
\smallskip
$$=\hspace{-1mm}1{-}\frax{1}{p}\hspace{-0,5mm}\left[1{-}\frax{(\a_1{+}\a_2{+}1)p^2{-}2(\a_1{+}\a_2{+}2)p{+}(\a_1{+}\a_2{+}1)}{(\a_2{-} \a_1{+}1)p^{\a_1{+}2}{-}(\a_2{-}\a_1{-}1)p^{\a_1{+}1}{-}(\a_1{+}\a_2{+}3)p{+}(\a_1{+}\a_2{+}1)}\right]\hspace{-1,5mm}=$$
\smallskip
$$\hspace{-17mm}=1{-}\frax{1}{p}\left[1{-}\frax{(n{+}1)p^2{-}2(n{+}2)p{+}(n{+}1)}{(n{-}2\a_1{+}1)p^{\a_1{+}2}{-}(n{-}2\a_1{-}1) p^{\a_1{+}1}{-}(n{+}3)p{+}(n{+}1)}\right].$$
\smallskip

\n The last expression is in this case a strictly decreasing
function in $\a_1$, which shows that
$\cdeg(\mathbb{Z}_{p^{\a_1}}\times\mathbb{Z}_{p^{\a_2}})$ is minimal
for $\a_1=[n/2]$ and maximal for $\a_1=1$. Hence we have proved the
following theorem.

\begin{theorem} In the class of abelian $p$-groups
$G$ of rank {\rm 2} and order $p^n$, we have that: \bit
\item[\rm a)] $|C(G)|$ is minimal {\rm(}maximal{\rm)} if and only if $G\cong\mathbb{Z}_p\times\mathbb{Z}_{p^{n-1}}$ {\rm(}respectively $G\cong\mathbb{Z}_{p^{[n/2]}}\times\mathbb{Z}_{p^{n-[n/2]}}${\rm)};
\item[\rm b)] $\cdeg(G)$ is minimal {\rm(}maximal{\rm)} if and only if $G\cong\mathbb{Z}_{p^{[n/2]}}\times\mathbb{Z}_{p^{n-[n/2]}}$ {\rm(}respectively $G\cong\mathbb{Z}_p\times\mathbb{Z}_{p^{n-1}}${\rm)}.
\eit
\end{theorem}

\bk\n{\bf Remark.} The above expression also shows that the
cyclicity degree of an abelian $p$-groups of rank {\rm 2} and
fixed order depends only on its number of subgroups.
\bigskip

\section{Further research} \label{concl}

Several questions and conjectures on cyclicity degrees of finite groups can be formulated. As an example we give the following

\bk\n{\bf Problem.} Is it true the following density result on
cyclicity degrees: for every $a \in [0,1]$ there exists a sequence
$(G_n)_{n \in \mathbb{N}}$ of finite groups such that
$\dd\lim_{n\to\infty} \cdeg(G_n)=a?$ Also, is it true that for every
$a \in (0,1] \cap \mathbb{Q}$ there exists a finite group $G$ such
that $\cdeg(G)=a$?


\vspace*{5ex}\small

\hfill
\begin{minipage}[t]{6cm}
Marius T\u arn\u auceanu \\
Faculty of  Mathematics \\
``Al.I. Cuza'' University \\
Ia\c si, Romania \\
e-mail: {\tt tarnauc@uaic.ro}
\end{minipage}

\vspace*{5ex}\small

\hfill
\begin{minipage}[t]{6cm}
L\'aszl\'o T\'oth \\
Department of Mathematics\\
University of P\'{e}cs\\
P\'ecs, Hungary \\
e-mail: {\tt ltoth@gamma.ttk.pte.hu}
\end{minipage}

\end{document}